\documentclass[12pt]{amsart}
\usepackage{amsfonts, amssymb, latexsym, epsfig} 

\usepackage{amscd,amssymb}
\usepackage{verbatim}
\usepackage{pstricks}
\usepackage{pst-node}
\usepackage[mathscr]{euscript}

\input xy
\xyoption{all}

\newsymbol\pp 1275
\usepackage{tikz}
\usetikzlibrary{matrix,arrows,backgrounds,shapes.misc,shapes.geometric,patterns,calc,positioning}

\usepackage[colorlinks=true, pdfstartview=FitV, linkcolor=blue, citecolor=blue, urlcolor=blue]{hyperref}

\usepackage{fullpage}
\usepackage{graphicx}
\usepackage{enumerate}
\def\VR{\kern-\arraycolsep\strut\vrule &\kern-\arraycolsep}
\def\vr{\kern-\arraycolsep & \kern-\arraycolsep}

\newtheorem{theorem}{Theorem}

\newtheorem{lemma}[theorem]{Lemma}
\newtheorem{prop}[theorem]{Proposition}

\theoremstyle{definition}
\newtheorem{definition}[theorem]{Definition}

\newtheorem{rmk}{Remark}
\newenvironment{remark}[1][]{\begin{rmk}[#1]\pushQED{\qed}}{\popQED \end{rmk}}

\newtheorem{obs}{Observation}
\newenvironment{observation}[1][]{\begin{obs}[#1]\pushQED{\qed}}{\popQED \end{obs}}

\newtheorem{ex}{Example}
\newenvironment{example}[1][]{\begin{ex}[#1]\pushQED{\qed}}{\popQED \end{ex}}

\newcommand{\leftsub}[2]{{\vphantom{#2}}_{#1}{#2}}

\newcommand{\Hom}{\operatorname{Hom}}
\newcommand{\End}{\operatorname{End}}
\newcommand{\Ext}{\operatorname{Ext}}

\newcommand{\ext}{\operatorname{ext}}
\newcommand{\rep}{\operatorname{rep}}
\newcommand{\Proj}{\operatorname{Proj}}
\newcommand{\gr}{\operatorname{gr}}
\newcommand{\SI}{\operatorname{SI}}
\newcommand{\SL}{\operatorname{SL}}
\newcommand{\GL}{\operatorname{GL}}
\newcommand{\PGL}{\operatorname{PGL}}

\newcommand{\ZZ}{\mathbb Z}

\newcommand{\PP}{\mathbb P}

\newcommand{\Ima}{\operatorname{Im}}
\newcommand{\Id}{\operatorname{Id}}

\newcommand{\Mat}{\operatorname{Mat}}

\newcommand{\Sch}{\operatorname{Schur}}

\newcommand{\ddim}{\operatorname{\mathbf{dim}}}
\newcommand{\dd}{\operatorname{\mathbf{d}}}
\newcommand{\ee}{\operatorname{\mathbf{e}}}
\newcommand{\hh}{\operatorname{\mathbf{h}}}

\newcommand{\pdim}{\mathsf{pdim}}
\newcommand{\C}{\mathcal{C}}

\newcommand{\R}{\operatorname{\mathcal{R}}}
\newcommand{\M}{\operatorname{\mathcal{M}}}

\newcommand{\module}{\operatorname{mod}}

\newcommand{\Schur}{\operatorname{Schur}}

\newcommand{\rk}{\operatorname{rank}}

\newcommand{\key}[1]{\emph{#1}}

\begin{document}
\title{Moduli spaces of modules of Schur-tame algebras}

\author{Andrew T. Carroll}
\address{University of Missouri-Columbia, Mathematics Department, Columbia, MO, USA}
\email[Andrew T. Carroll]{carrollat@missouri.edu}

\author{Calin Chindris}
\address{University of Missouri-Columbia, Mathematics Department, Columbia, MO, USA}
\email[Calin Chindris]{chindrisc@missouri.edu}

\date{\today}
\bibliographystyle{plain}

\begin{abstract} In this paper, we first show that for an acyclic gentle algebra $A$, the irreducible components of any moduli space of $A$-modules are products of projective spaces. Next, we show that the nice geometry of the moduli spaces of modules of an algebra does not imply the tameness of the representation type of the algebra in question. Finally, we place these results in the general context of moduli spaces of modules of Schur-tame algebras. More specifically, we show that for an arbitrary Schur-tame algebra $A$ and $\theta$-stable irreducible component $C$ of a module variety of $A$-modules, the moduli space $\M(C)^{ss}_{\theta}$ is either a point or a rational projective curve.
\end{abstract}

\maketitle
\setcounter{tocdepth}{1}
\tableofcontents

\section{Introduction}
Throughout, $K$ denotes an algebraically closed field of characteristic zero. All algebras are assumed to be bound quiver algebras, and all modules are assumed to be finite-dimensional left modules.

Our goal in this paper is to study the module category of an algebra $A$ within the general framework of geometric invariant theory. The geometric objects that we are interested in are the moduli spaces of semi-stable $A$-modules constructed by King in \cite{K}, using methods from geometric invariant theory. On the geometric side, these moduli spaces of modules can be arbitrarily complicated, in the sense that any projective variety can be realized as a moduli space of thin modules of some triangular algebra (see \cite{Hil}). On the representation theory side, the closed points of a moduli space of $A$-modules correspond to direct sums of rather special Schur $A$-modules. Hence, from the point of view of invariant theory, one is naturally led to think of an algebra based on the complexity of its Schur modules. In this paper, we focus on those algebras whose Schur modules have a tame behavior. These algebras, called Schur-tame, form a large class which goes beyond the class of tame algebras. Our objective is to describe the tameness, and more generally the Schur-tameness, of an algebra in terms of invariant theory. This line of research has attracted much attention during the last two decades (see for example  \cite{Bob1}, \cite{Bob5}, \cite{Bob6}, \cite{BS1}, \cite{Carroll1}, \cite{CC6}, \cite{CC9}, \cite{CC12}, \cite{Domo2}, \cite{GeiSch}, \cite{Rie}, \cite{Rie-Zwa-1}, \cite{Rie-Zwa-2}, \cite{SW1}).

A complete description of the tameness of quasi-tilted algebras in terms of their moduli spaces of modules can be found in  \cite{Bob5} and \cite{CC10}. In this paper, we first describe the irreducible components of all moduli spaces of modules for acyclic gentle algebras (for partial results, see \cite{CC13}). The indecomposable modules of gentle algebras can be nicely classified, however these tame algebras still represent an increase in the level of complexity from the tame quasi-tilted case.  For example, the global dimension of acyclic gentle algebras can be arbitrarily large. Furthermore, the number of one-parameter families required to parameterize $d$-dimensional indecomposable modules can grow faster than any polynomial in $d$.

\begin{theorem}\label{thm-main-1} Let $A=KQ/I$ be an acyclic gentle algebra, $\dd \in \ZZ^{Q_0}_{\geq 0}$ a dimension vector of $A$, and $\theta \in \ZZ^{Q_0}$ an integral weight such that $\dd$ is $\theta$-semi-stable. Then the irreducible components of the moduli space $\M(A,\dd)^{ss}_{\theta}$ are just products of projective spaces. 
\end{theorem} 

Our next theorem shows that the tameness of an algebra is not a reflection of the nice geometry of its moduli spaces. In particular, this provides an acyclic counterexample to (one of the implications in) Weyman's tameness conjecture (see \cite{CC13}). The algebra in the theorem below was communicated to us by Kinser and is based on Ringel's paper \cite{Rin2013}. 

\begin{theorem}\label{thm-main-2} Let $A=KQ/I$ be the wild Schur-tame algebra where:

$$
Q=\vcenter
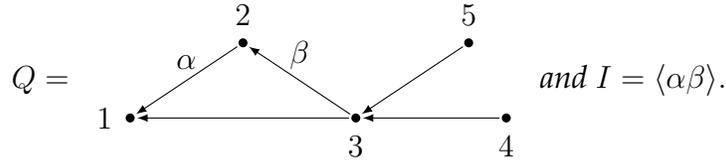
{\hbox{  
\begin{tikzpicture}[point/.style={shape=circle, fill=black, scale=.3pt,outer sep=3pt},>=latex]
   \node[point,label={left:$1$}] (1) at (0,0) {};
   \node[point,label={above:$2$}] (2) at (1.5,1) {};
   \node[point,label={below:$3$}] (3) at (3,0) {};
   \node[point,label={above:$5$}] (5) at (4.5,1) {};
   \node[point,label={below:$4$}] (4) at (5,0) {};   
  \path[->]
   (2) edge  node[midway, above] {$\alpha$} (1)
   (3) edge  node[midway, above] {$\beta$} (2)   
   (3) edge  (1)
   (5) edge  (3)
   (4) edge  (3);
   \end{tikzpicture} 
}} \text{~~~and~~~~} I=\langle \alpha\beta \rangle. 
$$

Let $\dd \in \ZZ^5_{\geq 0}$ be a dimension vector of $A$, $C \subseteq \module(A,\dd)$ an irreducible component, and $\theta \in \ZZ^5$ a weight of $A$ with $C^{ss}_{\theta} \neq \emptyset$. Then the moduli space $\M(C)^{ss}_{\theta}$ is a projective space.
\end{theorem}


In Section \ref{sec-background}, we outline the pertinent notions related to bound quiver algebras, module varieties, and Schur-tame algebras. In Section \ref{sec:moduli-spaces}, we first review King's construction of moduli spaces of modules of algebras, and then prove a general reduction result (Proposition \ref{prop-reduction}) that allows one to break a moduli space of modules into products of smaller ones. The proofs of our main results, presented in Section \ref{sec-proofs}, rely on descriptions of the irreducible components of the module varieties of the algebras involved (see Section \ref{gentle-sec}) and the general reduction result from  Section \ref{sec:moduli-spaces}. 

\subsection*{Acknowledgements} We would like to thank Ryan Kinser for communicating the wild Schur-tame algebra in Theorem \ref{thm-main-2} to us. We are also thankful to Jan Schr{\"o}er for clarifying discussions on preprojective algebras. The second author was supported by NSF grant DMS-1101383.

\section{Background}\label{sec-background}
\subsection{Bound quiver algebras}\label{sect:thebasics} Let $Q=(Q_0,Q_1,t,h)$ be a finite quiver with vertex set $Q_0$ and arrow set $Q_1$. The two functions $t,h:Q_1 \to Q_0$ assign to each arrow $a \in Q_1$ its tail \emph{ta} and head \emph{ha}, respectively.

A representation $V$ of $Q$ over $K$ is a collection $(V(x),V(a))_{x\in Q_0, a\in Q_1}$ of finite-dimensional $K$-vector spaces $V(x)$, $x \in Q_0$, and $K$-linear maps $V(a): V(ta) \to V(ha)$, $a \in Q_1$. The dimension vector of a representation $V$ of $Q$ is the function $\ddim V : Q_0 \to \ZZ$ defined by $(\ddim V)(x)=\dim_{K} V(x)$ for $x\in Q_0$. The one-dimensional representation of $Q$ supported at vertex $x \in Q_0$ is denoted by $S_x$ and its dimension vector is denoted by $\ee_x$. By a dimension vector of $Q$, we simply mean a vector $\dd \in \ZZ_{\geq 0}^{Q_0}$.

Given two representations $V$ and $W$ of $Q$, we define a morphism $\varphi:V \rightarrow W$ to be a collection $(\varphi(x))_{x \in Q_0}$ of $K$-linear maps with $\varphi(x) \in \Hom_K(V(x), W(x))$ for each $x \in Q_0$, and such that $\varphi(ha)V(a)=W(a)\varphi(ta)$ for each $a \in Q_1$. We denote by $\Hom_Q(V,W)$ the $K$-vector space of all morphisms from $V$ to $W$. Let $V$ and $W$ be two representations of $Q$. We say that $V$ is a subrepresentation of $W$ if $V(x)$ is a subspace of $W(x)$ for each $x \in Q_0$ and $V(a)$ is the restriction of $W(a)$ to $V(ta)$ for each $a \in Q_1$. In this way, we obtain the abelian category $\rep(Q)$ of all representations of $Q$.

Given a quiver $Q$, its path algebra $KQ$ has a $K$-basis consisting of all paths (including the trivial ones), and the multiplication in $KQ$ is given by concatenation of paths. It is easy to see that any $KQ$-module defines a representation of $Q$, and vice-versa. Furthermore, the category $\module(KQ)$ of $KQ$-modules is equivalent to the category $\rep(Q)$. In what follows, we identify $\module(KQ)$ and $\rep(Q)$, and use the same notation for a module and the corresponding representation.

A two-sided ideal $I$ of $KQ$ is said to be \emph{admissible} if there exists an integer $L\geq 2$ such that $R_Q^L\subseteq I \subseteq R_Q^2$. Here, $R_Q$ denotes the two-sided ideal of $KQ$ generated by all arrows of $Q$. 

If $I$ is an admissible ideal of $KQ$, the pair $(Q,I)$ is called a \emph{bound quiver} and the quotient algebra $KQ/I$ is called the \emph{bound quiver algebra} of $(Q,I)$. Bound quiver algebras are as general as they can be. Indeed, up to Morita equivalence, any finite-dimensional algebra $A$ can be viewed as the bound quiver algebra of a bound quiver $(Q_{A},I)$, where $Q_{A}$ is the Gabriel quiver of $A$ (see \cite[Corollary I.6.10 and Theorem~II.3.7]{AS-SI-SK}). (Note that the ideal of relations $I$ is not uniquely determined by $A$.) We say that $A$ is an \emph{acyclic} algebra if its Gabriel quiver has no oriented cycles.

Fix a bound quiver $(Q,I)$, a finite generating set $\R$ of admissible relations of $I$, and let $A=KQ/I$ be its bound quiver algebra.  A representation $M$ of $A$ (or  $(Q,I)$) is just a representation $M$ of $Q$ such that $M(r)=0$ for all $r \in \R$. The category $\module(A)$ of finite-dimensional left $A$-modules is equivalent to the category $\rep(A)$ of representations of $A$. As before, we identify $\module(A)$ and $\rep(A)$, and make no distinction between $A$-modules and representations of $A$.  For each vertex $x \in Q_0$, we denote by $P_i$ the projective indecomposable cover of the simple $A$-module $S_x$.  For an $A$-module $M$, we denote by $\pdim M$ its projective dimension.  An $A$-module $M$ is called \emph{Schur} if $\End_A(M) \cong K$. The dimension vector of a Schur $A$-module is called a \emph{Schur root} of $A$

Assume now that $A$ has finite global dimension; this happens, for example, when $Q$ has no oriented cycles. Then the \emph{Euler form} of $A$ is the bilinear form  $\langle \langle \cdot, \cdot \rangle \rangle_{A} : \ZZ^{Q_0}\times \ZZ^{Q_0} \to \ZZ$ defined by
$$
\langle \langle \dd,\ee \rangle \rangle_{A}=\sum_{l\geq 0}(-1)^l \sum_{x,y\in Q_0}\dim_K \Ext^l_{A}(S_x,S_y)\dd(x)\ee(y).
$$
In fact, for any $A$-modules $M$ and $N$ which are $\dd$- and
$\ee$-dimensional, respectively, one has
$$
\langle \langle \dd,\ee \rangle \rangle_{A}=\sum_{l\geq 0}(-1)^l \dim_K \Ext^l_{A}(M,N).
$$

\subsection{Module varieties and their irreducible components}\label{sect:modvar} Let $\dd$ be a dimension vector of $A=KQ/I$ (or equivalently, of $Q$). The affine variety
$$
\module(A,\dd):=\{M \in \prod_{a \in Q_1} \Mat_{\dd(ha)\times \dd(ta)}(K) \mid M(r)=0, \forall r \in
\R \}
$$
is called the \emph{module/representation variety} of $\dd$-dimensional modules/representations of $A$. The affine space $\module(Q,\dd):= \prod_{a \in Q_1} \Mat_{\dd(ha)\times \dd(ta)}(K)$ is acted upon by the base change group $$\GL(\dd):=\prod_{x\in Q_0}\GL(\dd(x),K)$$ by simultaneous conjugation, i.e., for $g=(g(x))_{x\in Q_0}\in \GL(\dd)$ and $V=(V(a))_{a \in Q_1} \in \module(Q,\dd)$,  $g \cdot V$ is defined by $$(g\cdot V)(a)=g(ha)V(a) g(ta)^{-1}, \forall a \in Q_1.$$ 

It can be easily seen that $\module(A,\dd)$ is a $\GL(\dd)$-invariant closed subvariety of $\module(Q,\dd)$, and that the $\GL(\dd)-$orbits in $\module(A,\dd)$ are in one-to-one correspondence with the isomorphism classes of the $\dd$-dimensional $A$-modules. 

In general, $\module(A, \dd)$ does not have to be irreducible. Let $C$ be an irreducible component of $\module(A, \dd)$. We say that $C$ is \emph{indecomposable} if $C$ has a non-empty open subset of indecomposable modules. We say that $C$ is a \emph{Schur component} if $C$ contains a Schur module. Obviously, any Schur component is indecomposable. A dimension vector $\dd$ is called a \emph{generic root} of $A$ if $\module(A,\dd)$ has an indecomposable irreducible component.

Given a decomposition $\dd=\dd_1+\ldots +\dd_l$ where $\dd_i \in \ZZ^{Q_0}_{\geq 0}, 1 \leq i \leq l$, and $\GL(\dd_i)$-invariant constructible subsets $C_i\subseteq \module(A,\dd_i)$, $1 \leq i \leq l$, we denote by $C_1\oplus \ldots \oplus C_l$ the constructible subset of $\module(A,\dd)$ defined as:  
$$C_1\oplus \ldots \oplus C_l=\{M \in \module(A,\dd) \mid M\simeq \bigoplus_{i=1}^t M_i\text{~with~} M_i \in C_i, \forall 1 \leq i \leq l\}.$$
As shown by de la Pe{\~n}a in \cite[Section 1.3]{delaP} and Crawley-Boevey and Schr{\"o}er in \cite[Theorem~1.1]{C-BS}, any irreducible component of a module variety has a Krull-Schmidt type decomposition. Specifically, if $C$ is an irreducible component of $\module(A,\dd)$ then there are unique generic roots $\dd_1, \ldots, \dd_l$ of $A$ such that $\dd=\dd_1+\ldots +\dd_l$ and
$$
C=\overline{C_1\oplus \ldots \oplus C_l}
$$
for some indecomposable irreducible components $C_i\subseteq \module(A,\dd_i), 1 \leq i \leq l$. Moreover, the indecomposable irreducible components $C_i, 1 \leq i \leq l,$ are uniquely determined by this property. We call $C=\overline{C_1\oplus \ldots \oplus C_l}$ the \emph{generic decomposition of $C$}.

Conversely, if $C_i \subseteq \module(A,\dd_i)$, $1 \leq i \leq l$, are indecomposable irreducible components then $\overline{C_1\oplus \ldots \oplus C_l}$ is an irreducible component of $\module(A,\sum_{i=1}^l \dd_i)$ if and only if $\ext_A^1(C_i,C_j)=0$ for all $1 \leq i \neq j \leq l$ (see \cite[Theorem~1.2]{C-BS}). Recall that if $D$ and $E$ are two irreducible components then $\ext_A^1(D,E):=\min \{\dim_K \Ext^1_A(X,Y) \mid (X,Y) \in D \times E\}$.

\subsection{Schur-tame algebras} Following Bodnarchuk-Drozd \cite{BodDro}, we now introduce the class of Schur-tame algebras. 

For an $A-R$-bimodule $T$, where $R$ is a localization $K[t]_f$ of $K[t]$ by a polynomial $f$, the functor $T\otimes_{R} -:\module(R) \to \module(A)$ is said to be a \emph{Schur-embedding} if: $(1)$ $T\otimes_R N \simeq T\otimes_R N'$ implies $N \simeq N'$; and $(2)$ $N$ is Schur then so is $T\otimes_R N$. Here, $\module(R)$ denotes the category of finite-dimensional $R$-modules. Also, recall that the finite-dimensional Schur $R$-modules are of the form ${K[t] \over (t-\lambda)}$ with $\lambda \in K$ such that $f(\lambda) \neq 0$ (for more details, see \cite[Ch. XIX.2]{Sim-Sko-3}).

\begin{definition} 
An algebra $A$ is said to be \emph{Schur-tame} if for each dimension vector $\dd$ of $A$, there are finitely many localizations $R_i=K[t]_{f_i}$, $1 \leq i \leq n_{\dd}$, and bimodules $\leftsub{A}(T_1)_{R_1},$ $ \ldots,$ $\leftsub{A}(T_{n_{\dd}})_{R_{n_{\dd}}}$ such that:
\begin{enumerate}
\item each $T_i$ is a free right $R_i$-module of finite rank and the functor $T_i \otimes_{R_i}-$ is a Schur-embedding;
\item every $\dd$-dimensional Schur $A$-module, except possibly for finitely many isoclasses of modules, is of the form $T_i\otimes_{R_i} {K[t] \over (t-\lambda)}$ for some $\lambda \in K$ with $f_i(\lambda) \neq 0$ and $1 \leq i \leq n_{\dd}$.
\end{enumerate}
\end{definition}

First, let us prove:

\begin{lemma} Any tame algebra $A$ is Schur-tame.
\end{lemma}

\begin{proof} Let $\dd$ be a dimension vector of $A$. We know that there are finitely many localizations $R_1, \ldots, R_n$ of $K[t]$ and bimodules  $\leftsub{A}(T_1)_{R_1}, \ldots, \leftsub{A}(T_{n_{\dd}})_{R_{n}}$ satisfying the two properties above with ``Schur-embedding'' replaced by ``representation-embedding'' in $(1)$, and ``Schur'' replaced by ``indecomposable'' in $(2)$. Following closely Dowbor-Skowro{\'n}ski's arguments in \cite{DowSko2}, we explain how to modify the $R_i$'s and $T_i$'s in order to get the desired Schur-embeddings that almost parametrize the $\dd$-dimensional Schur $A$-modules.

For each $1 \leq i \leq n$, write $R_i=K[t]_{f_i}$ and let $\mathcal U_i=(\mathbb A^1)_{f_i}$ be the principal open subset corresponding to $f_i$. Consider the morphism of varieties $\varphi_i:\mathcal U_i \to \module(A,\dd)$ induced by $T_i$; in particular, $\varphi_i(\lambda)\simeq T_i\otimes_{R_i} {K[t] \over (t-\lambda)}, \forall \lambda \in \mathcal U_i$. Next, note that the set $\Schur(A,\dd)$ consisting of all Schur $A$-modules in $\module(A,\dd)$ is an open subvariety due to the upper semi-continuity of the function $M \to \dim_K \End_A(M)$. Therefore, $\varphi_i^{-1}(\module(A,\dd) \setminus \Schur(A,\dd))$ must be a finite subset of $\mathcal U_i$; denote this subset by $\mathcal S_i$. Let:
\begin{itemize}
\item $f'_i:=\prod_{\alpha \in \mathcal S_i} (t-\alpha)$;
\item $R'_i:=(R_i)_{f'_i}$;
\item $T'_i:=T_i \otimes_{R_i} R'_i$.
\end{itemize}  
It is now clear that the set of the isoclasses of all modules of the form $T'_i \otimes_{R'_i} {K[t]\over (t-\lambda)}$, with $f'_i(\lambda) \neq 0$, is precisely the set of the isoclasses of the Schur $A$-modules of the form $T_i \otimes_{R_i} {K[t]\over (t-\lambda)}$, with $f_i(\lambda) \neq 0$; in particular, the functor $T'_i\otimes_{R'_i}-$ preserves Schur modules. Moreover, if $N$ and $N'$ are to finite-dimensional $R'_i$-modules such that $T'_i\otimes_{R'_i} N \simeq T'_i \otimes_{R'_i} N'$ then $T_i \otimes_{R_i} N \simeq T_i \otimes_{R_i} N'$ which implies that $N \simeq N'$ since $T_i \otimes_{R_i}-$ is a representation-embedding.

At this point, it is clear that the new localizations $R'_1, \ldots, R'_n$ and bimodules $T'_1, \ldots, T'_n$ satisfy the desired properties. We conclude that $A$ is Schur-tame. 
\end{proof}

\begin{example} 
\begin{enumerate}
\item [(i)] It has been pointed out to us by Schr{\"o}er that the preprojective algebra of a Dynkin quiver has only finitely many Schur modules in each dimension vector. That is to say, it is Schur-representation-finite and, in particular, Schur-tame.
\item [(ii)] For each integer $n \geq 4$, consider the bound quiver algebra given by
$$
Q=\vcenter
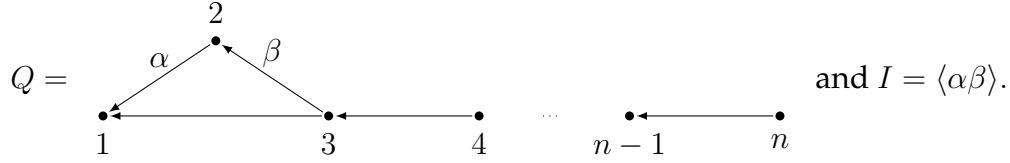
{\hbox{  
\begin{tikzpicture}[point/.style={shape=circle, fill=black, scale=.3pt,outer sep=3pt},>=latex]
   \node[point,label={below:$1$}] (1) at (0,0) {};
   \node[point,label={above:$2$}] (2) at (1.5,1) {};
   \node[point,label={below:$3$}] (3) at (3,0) {};
   \node[point,label={below:$4$}] (4) at (5,0) {};
   \node[point,label={below:$n-1$}] (5) at (7,0) {};
   \node[point,label={below:$n$}] (6) at (9,0) {};
    
   \draw[dotted] (5.85,0)--(6.1,0);   
      
  \path[->]
   (2) edge  node[midway, above] {$\alpha$} (1)
   (3) edge  node[midway, above] {$\beta$} (2)   
   (3) edge  (1)
   (4) edge  (3)
   (6) edge  (5);
   \end{tikzpicture} 
}} \text{~~~and~~~~} I=\langle \alpha\beta \rangle. 
$$
It was proved by Ringel in \cite{Rin2013} that these algebras are Schur-representation-finite, in particular Schur-tame, and that they are wild for $n \geq 9$.

\item [(iii)] Consider the algebra $A=KQ/I$ from Theorem \ref{thm-main-2}:
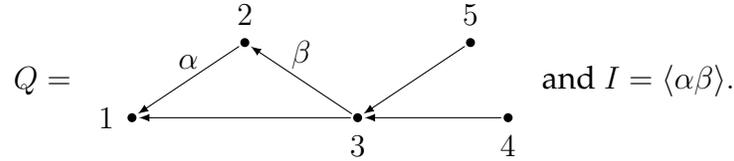
 $$Q=\vcenter{\hbox{  
\begin{tikzpicture}[point/.style={shape=circle, fill=black, scale=.3pt,outer sep=3pt},>=latex]
   \node[point,label={left:$1$}] (1) at (0,0) {};
   \node[point,label={above:$2$}] (2) at (1.5,1) {};
   \node[point,label={below:$3$}] (3) at (3,0) {};
   \node[point,label={above:$5$}] (5) at (4.5,1) {};
   \node[point,label={below:$4$}] (4) at (5,0) {};   
  \path[->]
   (2) edge  node[midway, above] {$\alpha$} (1)
   (3) edge  node[midway, above] {$\beta$} (2)   
   (3) edge  (1)
   (5) edge  (3)
   (4) edge  (3);
   \end{tikzpicture} 
}} \text{~~~and~~~~} I=\langle \alpha\beta \rangle. 
$$
It follows from Ringel's arguments in \cite{Rin2013} that $A$ is wild and Schur-tame. Specifically, to prove the wildness of $A$, one invokes a result of Mart{\'i}nez-Villa \cite{Martinez} to conclude that the non-simple indecomposable $A$-modules are in bijective correspondence with the non-simple indecomposable modules over the path algebra of the wild quiver obtained from $Q$ by splitting the vertex $2$ into two other vertices. Therefore, $A$ must be wild. As for the Schur-tameness of $A$, Ringel showed that for any Schur $A$-module $M$, either $M(\alpha)=0$ or $M(\beta)=0$. Consequently, the Schur modules for $A$ come from those for a $\tilde{\mathbb D}_4$ or $\mathbb D_5$ quiver. So, $A$ is Schur-tame.  
\end{enumerate}
\end{example}

For the remainder of this section, we assume that $A$ is a Schur-tame algebra and let $\dd$ be a Schur root of $A$. Denote by $\Sch(A,\dd)$ the open subvariety of $\module(A,\dd)$ consisting of all $\dd$-dimensional Schur $A$-modules. 

We know that there are finitely many principal open subsets $\mathcal U_i\subseteq \mathbb A^1=K$ and regular morphisms $\varphi_i:\mathcal U_i \to \module(A,\dd)$, $1 \leq i \leq n$, such that:
\begin{itemize}
\item for each $1 \leq i \leq n$, $\varphi_i(\mathcal U_i) \subset \Sch(A,\dd)$, and if $\varphi_i(\lambda_1) \simeq \varphi_i(\lambda_2)$ as $A$-modules then $\lambda_1=\lambda_2$;
\item all modules in $\Sch(A,\dd)$, except possibly finitely many isoclasses, belong to $\bigcup_{i=1}^n \mathcal F_i$, where each $\mathcal F_i$ is the closure of the image of the  action morphism $\GL(\dd) \times \mathcal U_i \to \module(A,\dd)$ that sends $(g,\lambda)$ to $g\cdot \varphi_i(\lambda)$, i.e. $\mathcal F_i=\overline{\bigcup_{\lambda \in \mathcal U_i} \GL(\dd)\varphi_i(\lambda)}$.
\end{itemize}
(We call $(\mathcal U_i,\varphi_i)$, $1 \leq i \leq n$, \emph{parameterizing pairs} for $\Sch(A,\dd)$.)

Consequently, we have that $$\overline{\Sch(A,\dd)}=\bigcup_{i=1}^n \mathcal F_i \cup \bigcup_{j=1}^l\overline{\GL(\dd)M_j}$$ for some $M_1, \ldots, M_l \in \Sch(A,\dd)$. 

Now, let $C \subseteq \module(A,\dd)$ be a Schur irreducible component; in particular, $C$ is an irreducible component of $\overline{\Sch(A,\dd)}$. From the discussion above, it follows that either:
\begin{itemize}
\item $C=\mathcal F_i$ for some $1 \leq i \leq n$ or;
\item $C=\overline{\GL(\dd)M_j}$ for some $1 \leq j \leq l$.
\end{itemize}

We have the following useful dimension count. For a proof, one can follow verbatim the arguments in \cite[Lemma 3]{CC13}.

\begin{lemma}\label{dim-count-irr-comp} Let $A$ be a Schur-tame algebra, $\dd$ a Schur root of $A$, and $C \subseteq \module(A,\dd)$ a Schur irreducible component. Then $\dim \GL(\dd)-\dim C \in \{0,1\}$, with $\dim \GL(\dd)=\dim C$ precisely when $C$ is not an orbit closure.
\end{lemma}

\begin{rmk} We point out that a strictly wild algebra $A$ can not be Schur-tame. Indeed, it follows from \cite[Section 1.5]{delaP} that for any positive integer $n$, there exist a Schur root $\dd$ of $A$ and a Schur irreducible component $C\subseteq \module(A,\dd)$ such that $\dim \GL(\dd)-\dim C \leq -n^2$. This inequality combined with Lemma \ref{dim-count-irr-comp} shows that $A$ is not Schur-tame. 
\end{rmk}


\section{Theta-stable decompositions and moduli spaces}\label{sec:moduli-spaces}
Let $A=KQ/I$ be an algebra and let $\dd \in \ZZ^{Q_0}_{\geq 0}$ be a dimension vector of $A$. Let us consider the subgroup $\SL(\dd):=\prod_{x \in Q_0}\SL(\dd(x),K)$ of $\GL(\dd)$ and its action on $K[\module(A,\dd)]$. The resulting ring of semi-invariants $\SI(A,\dd):=K[\module(A,\dd)]^{\SL(\dd)}$ has a weight space decomposition over the group $X^{\star}(\GL(\dd))$ of rational characters of $\GL(\dd)$:
$$\SI(A,\dd)=\bigoplus_{\chi \in X^\star(\GL(\dd))}\SI(A,\dd)_{\chi}.
$$
For each character $\chi \in X^{\star}(\GL(\dd))$, $$\SI(A,\dd)_{\chi}=\lbrace f \in K[\module(A,\dd)] \mid g f= \chi(g)f \text{~for all~}g \in \GL(\dd)\rbrace$$ is called the \key{space of semi-invariants} on $\module(A,\dd)$ of \key{weight} $\chi$. For a $\GL(\dd)$-invariant closed subvariety $C \subseteq \module(A,\dd)$, we similarly define the ring of semi-invariants $\SI(C):=K[C]^{\SL(\dd)}$, and the space $\SI(C)_{\chi}$ of semi-invariants of weight $\chi$.

Note that any $\theta \in \ZZ^{Q_0}$ defines a rational character $\chi_{\theta}:\GL(\dd) \to K^*$ by 
\begin{equation}
\chi_{\theta}((g(x))_{x \in Q_0})=\prod_{x \in Q_0}\det g(x)^{\theta(x)}.
\end{equation}
In this way, we identify $\ZZ ^{Q_0}$ with $X^{\star}(\GL(\dd))$ whenever $\dd$ is a sincere dimension vector. In general, we we have the natural epimorphism $\ZZ^{Q_0} \to X^{\star}(\GL(\dd))$. We also refer to the rational characters of $\GL(\dd)$ as (integral) weights of $A$ (or $Q$).

Following King \cite{K}, an $A$-module $M$ is said to be \emph{$\theta$-semi-stable} if $\theta(\ddim M)=0$ and $\theta(\ddim M')\leq 0$ for all submodules $M' \leq M$. We say that $M$ is \emph{$\theta$-stable} if $M$ is non-zero, $\theta(\ddim M)=0$, and $\theta(\ddim M')<0$ for all submodules $\{0\} \neq M' < M$. A \key{$\theta$-polystable} $A$-module is defined to be a direct sum of $\theta$-stable $A$-modules. The full subcategory $\module(A)^{ss}_{\theta}$ consisting of the $\theta$-semi-stable $A$-modules is an exact abelian subcategory of $\module(A)$ which is closed under extensions and whose simple objects are precisely the $\theta$-stable modules. Moreover, $\module(A)^{ss}_{\theta}$ is Artinian and Noetherian, and hence every $\theta$-semi-stable $A$-module $M$ has a Jordan-H{\"o}lder filtration in $\module(A)^{ss}_{\theta}$; the direct sum of the factors of such a filtration of $M$ is a $\theta$-polystable $A$-module which, up to isomorphism, is independent of the Jordan-H{\"o}lder filtration. 

Now, let us consider the (possibly empty) open subsets
$$\module(A,\dd)^{ss}_{\theta}=\{M \in \module(A,\dd)\mid M \text{~is~}
\text{$\theta$-semi-stable}\}$$
and $$\module(A,\dd)^s_{\theta}=\{M \in \module(A,\dd)\mid M \text{~is~}
\text{$\theta$-stable}\}$$
of $\dd$-dimensional $\theta$(-semi)-stable $A$-modules. Using methods from Geometric Invariant Theory, King showed in \cite{K} that the projective variety
$$
\M(A,\dd)^{ss}_{\theta}:=\Proj(\bigoplus_{n \geq 0}\SI(A,\dd)_{n\theta})
$$
is a GIT-quotient of $\module(A,\dd)^{ss}_{\theta}$ by the action of $\PGL(\dd)$ where $\PGL(\dd)=\GL(\dd)/T_1$ and $T_1=\{(\lambda \Id_{\dd(x)})_{x \in Q_0} \mid \lambda \in k^*\} \leq \GL(\dd)$. Moreover, there is a (possibly empty) open subset $\M(A,\dd)^s_{\theta}$ of $\M(A,\dd)^{ss}_{\theta}$ which is a geometric quotient of $\module(A,\dd)^s_{\theta}$ by $\PGL(\dd)$. We say that $\dd$ is a \emph{$\theta$-(semi-)stable dimension vector} of $A$ if $\module(A,\dd)^{(s)s}_{\theta} \neq \emptyset$. 

For a $\GL(\dd)$-invariant closed subvariety $C \subseteq \module(A,\dd)$, we similarly define $C^{ss}_{\theta}, C^s_{\theta}$, $\M(C)^{ss}_{\theta}$, and $\M(C)^s_{\theta}$. Note that if $\pi:\module(A,\dd)^{ss}_{\theta} \to \M(A,\dd)^{ss}_{\theta}$ is the quotient morphism then $\M(C)^{ss}_{\theta}$ is precisely $\pi(C^{ss}_{\theta})$. We say that $C$ is a \emph{$\theta$-(semi-)stable subvariety} if $C^{(s)s} \neq \emptyset$. It was proved by King that the closed points of $\M(C)^{ss}_{\theta}$ correspond bijectively to the $\theta$-polystable $A$-modules in $C$ and that, for any $M \in C^{ss}_{\theta}$, $\overline{\GL(\dd)M}$ contains a unique, up to isomorphism, $\theta$-polystable $A$-module; in fact, this $\theta$-polystable module, which we denote by $\gr_{\theta}(M)$, is nothing else but (isomorphic to) the direct sum of the factors of a Jordan-H{\"o}lder filtration of $M$ in $\module(A)^{ss}_{\theta}$. Furthermore, the $\GL(\dd)$-orbit of $\gr_{\theta}(M)$ is closed in $C^{ss}_{\theta}$, being the unique such closed orbit contained in $\overline{\GL(\dd)M} \cap C^{ss}_{\theta}$, and $\pi(M)=\pi(\gr_{\theta}(M))$.

We now come to the key concept of this section. Let $C$ be a $\GL(\dd)$-invariant irreducible closed subvariety of $\module(A,\dd)$. Assume that $C$ is $\theta$-semi-stable and let $\dd_1, \ldots, \dd_l$ be $\theta$-stable dimension vectors of $A$ with $\sum_{i=1}^l \dd_i=\dd$. Let $C_i \subseteq \module(A,\dd_i)$, for $1 \leq i \leq l$, be a $\theta$-stable $\GL(\dd_i)$-invariant closed subvariety. We write
$$
C=C_1\pp \ldots \pp C_l,
$$
to mean that for the generic module $M \in C^{ss}_{\theta}$, $\gr_{\theta}(M)$ is (isomorphic to a module) in $C^s_{1, \theta}\oplus \ldots \oplus C^s_{l, \theta}$. We call such a decomposition of $C$, whenever it exists, a \key{$\theta$-stable decomposition}. 

\begin{remark} It follows from the work of Bobi{\'n}ski and Skowro{\'n}ski \cite{BS1} that for a tame quasi-tilted algebra, any $\theta$-semi-stable irreducible component is $\theta$-well-behaved (in the sense of \cite{CC13}); in particular, it has a unique $\theta$-stable decomposition (for details, see \cite{CC13}). The same holds for acyclic gentle algebras (see Section \ref{gentle-sec} or \cite{CC13}). 
\end{remark}

Now, we are ready to state the following reduction theorem from \cite[Theorem 1.4]{CC10}:

\begin{theorem}\label{theta-stable-decomp-thm} Let $A$ be an algebra, $\dd$ a dimension vector of $A$, and $\theta$ an integral weight of $A$. Let $C$ be a normal $\GL(\dd)$-invariant closed subvariety of $\module(A,\dd)$ that admits a $\theta$-stable decomposition:
$$C=m_1\cdot C_1 \pp \ldots \pp m_l\cdot C_l,$$
where $m_1, \ldots, m_l \geq 1$, $C_i \subseteq \module(A,\dd_i)$, $1 \leq i \leq l$, are $\theta$-stable irreducible components, and $\dd_i\neq \dd_j$ for all $1 \leq i\neq j \leq l$. Furthermore, assume that $\bigoplus_{i=1}^l C_i^{\oplus m_i} \subseteq C$. Then: 
$$
\mathcal{M}(C)^{ss}_{\theta} \cong S^{m_1}(\mathcal{M}(C_1)^{ss}_{\theta}) \times \ldots \times S^{m_l}(\mathcal{M}(C_l)^{ss}_{\theta}).
$$
\end{theorem}

Note that this reduction result allows us to ``break'' a moduli space of modules into smaller ones which are easier to handle, especially in the Schur-tame case. 

The next result is a strengthening of the reduction Theorem \ref{theta-stable-decomp-thm} in that it allows us to get rid of the orbit closures that occur in a $\theta$-stable decomposition. It plays a crucial role in proving Theorems \ref{thm-main-1} and \ref{thm-main-2}.

\begin{prop}\label{prop-reduction} Let $A$ be an algebra, $\dd$ a dimension vector of $A$, and $\theta$ an integral weight of $A$. Let $C$ be a $\theta$-semi-stable irreducible component of $\module(A,\dd)$ such that $\M(C)^{ss}_{\theta}$ is an irreducible component of $\M(A,\dd)^{ss}_{\theta}$. Assume that:
\begin{enumerate} 
\item $C$ is normal;
\item $C$ admits a $\theta$-stable decomposition of the form:
$$C=m_1\cdot C_1 \pp \ldots \pp m_l\cdot C_l \pp C_{l+1} \ldots \pp C_n ,$$
where $m_1, \ldots, m_l \geq 1$, $C_i \subseteq \module(A,\dd_i)$, $1 \leq i \leq n$, are $\theta$-stable irreducible components with $C_{l+1}, \ldots C_n$ orbit closures, and $\dd_i\neq \dd_j$ for all $1 \leq i\neq j \leq l$;
\item $C':=\overline{\bigoplus_{i=1}^l C_i^{\oplus m_i}}$ is a normal subvariety of $\module(A, \sum_{i=1}^l m_i \dd_i)$. 
\end{enumerate}
Then:
$$
\mathcal{M}(C)^{ss}_{\theta} \cong S^{m_1}(\mathcal{M}(C_1)^{ss}_{\theta}) \times \ldots \times S^{m_l}(\mathcal{M}(C_l)^{ss}_{\theta}).
$$

\begin{proof} We know that $C_j=\overline{\GL(\dd_j)M_j}$ with $M_j \in \module(A,\dd_j)^s_{\theta}$ for $l+1 \leq j \leq n$. Set $M_0=M_{l+1} \oplus \ldots \oplus M_n$ and note that $M_0$ is $\theta$-polystable. 

Let $\pi:\module(A,\dd)^{ss}_{\theta} \to \M(A,\dd)^{ss}_{\theta}$ and $\pi':(C')^{ss}_{\theta} \to \M(C')^{ss}_{\theta}$ be the quotient morphisms. Furthermore, consider the morphism $\varphi:(C')^{ss}_{\theta} \to \module(A,\dd)^{ss}_{\theta}$ defined by $\varphi(X)=X\oplus M_0$, for all $X \in (C')^{ss}_{\theta}$. From the universal property of the GIT quotient $\M(C')^{ss}_{\theta}$, we get the commutative diagram:
$$
\vcenter{\hbox{  
\begin{tikzpicture}
  \matrix (m) [matrix of math nodes,row sep=3em,column sep=4em,minimum width=2em]
  {
     (C')^{ss}_{\theta} & \M(C')^{ss}_{\theta} \\
     \module(A,\dd)^{ss}_{\theta} &  \\
     \M(A,\dd)^{ss}_{\theta} & \\};
  \path[-stealth]
    (m-1-1) edge node [above] {$\pi'$} (m-1-2)
    (m-1-1) edge node [left] {$\varphi$} (m-2-1)
    (m-2-1) edge node [left] {$\pi$} (m-3-1)
    (m-1-2) edge [dashed, ->] node [right] {$f$} (m-3-1);
\end{tikzpicture}
}}
$$
where $f:\M(C')^{ss}_{\theta} \to \M(A,\dd)^{ss}_{\theta}$ is the
morphism of varieties defined so that $f(\pi'(X))=\pi(X \oplus M_0)$
for all $X \in (C')^{ss}_{\theta}$. Let us denote $\Ima(f)$ by $Y$. We
claim that $Y=\M(C)^{ss}_{\theta}$. Indeed, the $\theta$-stable
decomposition of $C$ simply says that for a generic point $\tilde{X} \in C^{ss}_{\theta}$, $\gr_{\theta}(\tilde{X})$ is of the form $X\oplus M_0$ for some $X \in (C')^{ss}_{\theta}$. So, the generic point of $\M(C)^{ss}_{\theta}$ is of the form $f(\pi'(X)) \in Y$ and hence $\M(C)^{ss}_{\theta} \subseteq Y$. But this clearly implies our claim since $Y$ is irreducible and $\M(C)^{ss}_{\theta}$ is assumed to be an irreducible component of $\M(A,\dd)^{ss}_{\theta}$. 

In what follows, we show that $f:\M(C')^{ss}_{\theta} \to \M(C)^{ss}_{\theta}=Y$ is an isomorphism of varieties. First, let us check that $f$ is bijective. Since $f$ is surjective, we proceed with checking the injectivity of $f$. Let $x,y \in \M(C')^{ss}_{\theta}$ be so that $f(x)=f(y)$. Choose $\theta$-polystable $A$-modules $X,Y \in (C')^{ss}_{\theta}$ such that $\pi'(X)=x$ and $\pi'(Y)=y$. Then, $f(x)=f(y)$ is equivalent to $\pi(X\oplus M_0)=\pi(Y\oplus M_0)$ which is further equivalent to $X\oplus M_0 \simeq Y\oplus M_0$ since these two direct sums are still $\theta$-polystable. We conclude that $X\simeq Y$, and hence $x=y$. So, $f$ is injective. 

Since $C$ is assumed to be normal, the GIT quotient $\M(C)^{ss}_{\theta}$ remains a normal variety. We have just proved that $f:\M(C')^{ss}_{\theta} \to \M(C)^{ss}_{\theta}$ is a bijective morphism with normal target variety. Therefore, $f$ has to be an isomorphims of varieties. (Here, we are using again the assumption that $K$ is of characteristic zero.) The proof now follows from Theorem \ref{theta-stable-decomp-thm}.
\end{proof}
\end{prop}

\section{Proofs of the main results}\label{sec-proofs}

\subsection{Acyclic gentle algebras} \label{gentle-sec}
We will review basic definitions and key facts concerning acyclic gentle algebras before proving Theorem \ref{thm-main-1}.  Recall that an algebra $A$ is called gentle if it is isomorphic to a bound quiver algebra $KQ/I$ satisfying the following:
\begin{enumerate}
\item for each vertex $x\in Q_0$ there are at most two arrows with
  head $x$, and at most two arrows with tail $x$;
\item for any arrow $b\in Q_1$, there is at most one arrow $a\in Q_1$ and at most one arrow $c\in Q_1$ such that $ab\notin I$ and $bc\notin I$;
\item for each arrow $b\in Q_1$ there is at most one arrow $a\in Q_1$ with $ta=hb$ (resp. at most one arrow $c\in Q_1$ with $hc=tb$) such that $ab\in I$ (resp. $bc\in I$);
\item $I$ is generated by paths of length 2.
\end{enumerate}
     
In \cite{Carroll1}, a combinatorial characterization of the irreducible components of module varieties for these algebras was obtained. By a \emph{coloring} of a quiver $Q$, we mean a map $c: Q_1 \rightarrow S$ (where $S$ is some finite set) such that $c^{-1}(s)$ is a directed path for each $s\in S$. For a coloring of $Q$, we define by $I_c$ the two-sided ideal in $KQ$ generated by all length-two paths $ba$ for which $c(a)=c(b)$. Furthermore, for every acyclic gentle algebra $KQ/I$, there is a coloring $c$ of $Q$ for which $I=I_c$.

Fix a gentle algebra $A=KQ/I$ and a coloring $c$ for which $I=I_c$. A \emph{rank sequence} for a dimension vector $\dd\in \ZZ^{Q_0}_{\geq 0}$ is a map $r: Q_1 \rightarrow \ZZ_{\geq 0}$ satisfying the property that $r(a)+r(b) \leq \dd_x$ whenever $c(a)=c(b)$, and $h(a)=t(b)=x$ (together with the degenerate condition $r(a) \leq \dd_x$ when $x$ is a source or sink and $a$ is any arrow incident to it).  

\begin{prop}[\cite{Carroll1}] The irreducible components of
  $\module(A, \dd)$ are parameterized by rank sequences $r$ for $\dd$
  which are maximal relative to the coordinate-wise partial order.  In
  particular, the irreducible components are of the form \[ \C(A, \dd,
  r) = \{V\in \module(A, \dd) \mid \operatorname{rank}_K V(a) \leq
  r(a)\}\] for $r$ maximal.
\end{prop}
\noindent
As a consequence, the irreducible components of $\module(A, \dd)$ are products of varieties of complexes, and are therefore normal (see \cite{DeConcini-Strickland}).  

Gentle algebras are a special class of string algebras, whose indecomposable modules are known to be either string modules or band modules (see \cite{BR}).  We call an irreducible component $C\subseteq \module(A, \dd)$ \emph{regular} if the generic module in $C$ is a direct sum of band modules.

\begin{observation} Suppose that $A$ is acyclic gentle.  An irreducible component $C\subseteq \module(A,\dd)$ is regular if and only if it contains a module which is a direct sum of band modules. Indeed, consider the open non-empty subvariety of $C$:
$$
\mathcal U=\{M \in C \mid \rk M(a)=\max \{\rk X(a) \mid X \in C\}, \forall a \in Q_1\}.
$$
Let $M_0 \in C$ be a regular module. Then, for any $M \in \mathcal U$, we have: 
$$
\dim_K M_0=\sum_{a \in Q_1} \rk M_0(a) \leq \sum_{a \in Q_1} \rk M(a)=\dim_k M-s,
$$
where $s$ is the number of string indecomposable modules occurring in a direct sum decomposition of $M$ into indecomposables. Consequently, $s=0$ and hence the generic modules in $C$, more precisely those in $\mathcal{U}$, are regular. 
\end{observation}

\begin{prop}  \label{prop-reg-ind}
\begin{enumerate}
\item (\cite{Carroll2})  Suppose that $C$ is an indecomposable regular irreducible component, then the generic module $M$ of $C$ is Schur and $\pdim M \leq 1$.  
\item (\cite{CC13}) Any module variety for the gentle algebra $A$ has at most one regular irreducible component. Furthermore, if $C$ is the regular irreducible component of some $\module(A,\dd)$ then $\ext_A^1(C,C)=0$ and $\langle \langle \dd, \dd \rangle \rangle= 0$.
\end{enumerate} 
\end{prop}
\noindent
From this proposition, we deduce that given two stable (with respect to some weight) irreducible components of the same module variety $\module(A,\dd)$, their stable loci are disjoint. In particular, for an acyclic gentle algebra, any $\theta$-semi-stable irreducible component  is $\theta$-well-behaved and it has therefore a unique $\theta$-stable decomposition.

We are now ready to prove Theorem \ref{thm-main-1}.  

\begin{proof}[Proof of Theorem \ref{thm-main-1}] Let $Y$ be an irreducible component of $\M(A,\dd)^{ss}_{\theta}$. Then $Y=\M(C)^{ss}_{\theta}$ for some $\theta$-semi-stable irreducible component $C \subseteq \module(A,\dd)$. We have seen that $C$ has a $\theta$-stable decomposition of the form:
$$
C=m_1\cdot C_1 \pp \ldots \pp m_l\cdot C_l \pp C_{l+1} \ldots \pp C_n ,$$
where $m_1, \ldots, m_l \geq 1$, $C_i \subseteq \module(A,\dd_i)$, $1 \leq i \leq l$, are $\theta$-stable regular irreducible components, $C_{l+1}, \ldots C_n$ are orbit closures, and $\dd_i\neq \dd_j$ for all $1 \leq i\neq j \leq l$. Moreover, we know that $\langle \langle \dd_i, \dd_i \rangle \rangle=\ext_A^1(C_i,C_i)=0$ for all $1 \leq i \leq l$. Denote $\sum_{i=1}^l m_i \dd_i$ by $\dd'$ and $\overline{\bigoplus_{i=1}^l C_i^{\oplus m_i}} \subseteq \module(A,\dd')
$ by $C'$. Note that a $\theta$-stable decomposition of $C'$ is $m_1\cdot C_1 \pp \ldots \pp m_l\cdot C_l$. 

We show next that $C'$ is an irreducible component of $\module(A,\dd')$ by checking that, for all $1 \leq i,j \leq l$, $\ext^1_A(C_i,C_j)=0$; in particular, this will prove that $C'$ is normal. Choose $A$-modules $M_i \in (C_i)^s_{\theta}$ with $\pdim M_i \leq 1$ for all $1\leq i \leq l$. Then, for all $1 \leq i \neq j \leq l$, $\Hom_A(M_i,M_j)=0$ since $M_i$ and $M_j$ are non-isomorphic $\theta$-stable modules, and hence: 
$$
\langle \langle \dd_i, \dd_j \rangle \rangle=-\dim_K \Ext^1_A(M_i,M_j).
$$
Consequently, we get that
$$
0=\langle \langle \dd,\dd \rangle \rangle=\sum_{1 \leq i \neq j \leq l} -m_i m_j \dim_K \Ext^1_A(M_i,M_j)
$$
which shows that $\ext_A^1(C_i,C_j)=0$ for all $1 \leq i,j \leq l$. At this point, we can apply Proposition \ref{prop-reduction} to conclude that
$$
Y=\M(C)^{ss}_{\theta} \simeq \prod_{i=1}^l \mathbb{P}^{m_i}.
$$ 
\end{proof}

We turn now to the more general acyclic string algebras. 

\begin{prop} \label{cor-string} Let $A=KQ/I$ be an acyclic string algebra, $\dd\in \mathbb{Z}_{\geq 0}^{Q_0}$ a dimension vector of $A$, and $\theta\in \mathbb{Z}^{Q_0}$ an integral weight such that $\dd$ is $\theta$-semi-stable. If $\M(C)_\theta^{ss}$ is an irreducible component of $\M(A, \dd)_\theta^{ss}$ for which $C$ is normal, then $\M(C)_\theta^{ss}$ is a product of projective spaces.  
\end{prop}

We will first require a combinatorial lemma concerning the ideal of relations of an acyclic string algebra:

\begin{lemma} \cite{Carroll2} \label{lemma-string-gentle} Let $A=KQ/I$ be an acyclic string algebra. Then there exists a coloring $c$ of $Q$ such that $I_c \subseteq I$. In particular, any acyclic string algebra is a quotient of an acyclic gentle algebra.   
\end{lemma}

\begin{proof}[Proof of Proposition \ref{cor-string}] First, we check that any module variety $\module(A,\dd)$ has at most one Schur regular irreducible component. This will ensure that any $\theta$-semi-stable irreducible component is $\theta$-well-behaved and, hence, has a (unique) $\theta$-stable decomposition.

So, let $C_0 \subseteq \module(A, \dd)$ be a Schur regular irreducible component and let $M_0 \in C$ be a Schur regular $A$-module. From Lemma \ref{lemma-string-gentle}, we know that there is a gentle algebra $\tilde{A}=KQ/\tilde{I}$ with $\tilde{I} \subseteq I$. Clearly, $M_0$ is a Schur regular module over $\tilde{A}$ as well. Let $\tilde{C}_0 \subseteq \module(\tilde{A},\dd)$ be an irreducible component such that $C_0 \subseteq \tilde{C}_0$. Then $\tilde{C}_0$ is a Schur regular irreducible component of $\module(\tilde{A}, \dd)$ and, moreover, $\dim C_0=\dim \tilde{C}_0=\dim \GL(\dd)$ by Lemma \ref{dim-count-irr-comp}. Consequently $C_0=\tilde{C}_0$, i.e $C_0$ is the unique Schur regular irreducible component of $\module(\tilde{A},\dd)$. 

Now, let $C$ be a normal, $\theta$-semi-stable irreducible component of $\module(A, \dd)^{ss}_\theta$. From the discussion above, $C$ has a $\theta$-stable decomposition given by \[
  m_1\cdot C_1\pp \dotsc \pp m_l\cdot C_l \pp C_{l+1} \pp \dotsc \pp C_n\] with
  $C_i \subset \module(A, \dd_i)$ such that $C_i$ is regular $\theta$-stable for
  $i=1,\dotsc, l$, $\dd_i \neq \dd_j$ for $1 \leq i \neq j \leq l$, and $C_{l+1}, \dotsc, C_n$ are orbit closures.  Moreover, $\overline{\bigoplus\limits_{i=1}^l C_i^{\oplus m_i}}$ becomes an irreducible component of $\module(\tilde{A},\sum_{i=1}^l m_i \dd_i)$ and is therefore normal. Applying Proposition \ref{prop-reduction} again, we conclude that $\M(C)^{ss}_{\theta}$ is a product of projective spaces. 
\end{proof}

\subsection{Wild Schur-tame algebras} We now turn our attention to the wild Schur-tame algebra $A=KQ/I$ in Theorem \ref{thm-main-2} which is given by:
$$Q=\vcenter
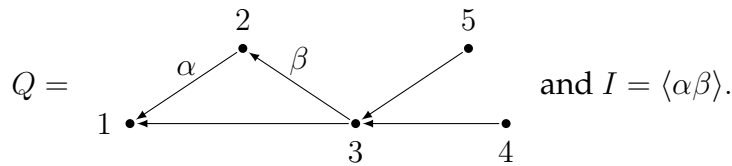
{\hbox{  
\begin{tikzpicture}[point/.style={shape=circle, fill=black, scale=.3pt,outer sep=3pt},>=latex]
   \node[point,label={left:$1$}] (1) at (0,0) {};
   \node[point,label={above:$2$}] (2) at (1.5,1) {};
   \node[point,label={below:$3$}] (3) at (3,0) {};
   \node[point,label={above:$5$}] (5) at (4.5,1) {};
   \node[point,label={below:$4$}] (4) at (5,0) {};   
  \path[->]
   (2) edge  node[midway, above] {$\alpha$} (1)
   (3) edge  node[midway, above] {$\beta$} (2)   
   (3) edge  (1)
   (5) edge  (3)
   (4) edge  (3);
   \end{tikzpicture} 
}} \text{~~~and~~~~} I=\langle \alpha\beta \rangle. 
$$

\begin{proof}[Proof Theorem \ref{thm-main-2}] Let $\dd$ be a dimension vector of $A$, $\theta$ a weight, and $C \subseteq \module(A,\dd)$ a $\theta$-semi-stable irreducible component. We distinguish two cases:

\vspace{10pt}
\noindent
\textsf{Case 1}: $\theta(2)=0$. Denote by $\mathbb D_4$ the subquiver of $Q$ obtained by deleting vertex $2$ and the two arrows incident to it,  and denote by $\dd'$ the restriction of $\dd$ to $\mathbb D_4$. 

Let $X=\{(M(\alpha), M(\beta)) \in \Mat_{\dd(1) \times \dd(2)} \times \Mat_{\dd(2)\times \dd(3)} \mid M(\alpha) \cdot M(\beta)=0\}$. Then
$$
R:=K[\module(A,\dd)]^{\GL(\dd(2))}=K[X]^{\GL(\dd(2))} \otimes_K K[\module(\mathbb{D}_4,\dd')].
$$
Using the First Fundamental Theorem for $\GL(\dd(2))$, we get that $K[X]^{\GL(\dd(2))}=K$ and hence $R=K[\module(\mathbb{D}_4,\dd')]$. Moreover, if $\theta'$ is the restriction of $\theta$ to $\mathbb{D}_4$ then
$$
\dim_K \SI(A,\dd)_{l \theta}=\dim_K \SI(\mathbb{D}_4,\dd')_{l \theta'} \leq 1, \forall l \geq 1.
$$
It now follows that $\dim_K \SI(C)_{l \theta} \leq 1$ for all $l \geq 1$, i.e. $\M(C)^{ss}_{\theta}$ is just a point when $\theta(2)=0$.

\vspace{25pt}
\noindent
\textsf{Case 2}: $\theta(2) \neq 0$. Note first that for any dimension vector $\hh$ of $A$ the irreducible components of $\module(A,\hh)$ are of the form:
$$
C(r_{\alpha}, r_{\beta}):=\{M \in \module (A, \hh) \mid \rk M(\alpha) \leq r_{\alpha}, \rk M(\beta) \leq r_{\beta}\},
$$
where $(r_{\alpha}, r_{\beta})$ is a maximal (coordinatewise) pair of non-negative integers with $r_{\alpha}+r_{\beta} \leq \hh(2)$. So, they are all normal varieties.

In what follows, we denote by $\mathbb D_5$ the subquiver of $Q$ obtained by deleting only the arrow $\beta$, and by $\widetilde{\mathbb D}_4$ the subquiver of $Q$ obtained by deleting only the arrow $\alpha$. From Ringel's description of the Schur $A$-modules, we know that the Schur components of a module variety $\module(A,\hh)$ are of the form $\module(\mathbb D_5, \hh)$ or $\module(\widetilde{\mathbb D}_4, \hh)$. In particular, this shows that if $C'$ is a $\theta$-stable irreducible component of some module variety $\module(A,\hh)$ then $\M(C')^{ss}_{\theta}$ is either a point or $\mathbb P^1$ (see for example \cite{CC10}).

Next, we claim that a module variety $\module(A,\hh)$, with $\hh$ a $\theta$-semi-stable dimension vector of $A$, can have at most one $\theta$-stable irreducible component. Since $\theta(2)\neq 0$, we get that $\hh \neq \ee_2$. Furthermore, if $\hh(2)=0$ then $\module(A,\hh)=\module(\mathbb{D}_4, \hh)$ which is an affine space, so there is nothing to check in this case. Let us assume now that $\hh(2)\geq 1$. In this case, we only need to check that $\module(\mathbb D_5, \hh)^s_{\theta}$ and $\module(\widetilde{\mathbb D}_4, \hh)^s_{\theta}$ can not be both non-empty. Let us assume for a contradiction that the two $\theta$-stable loci above are non-empty. Since vertex $2$ is a sink for $\tilde{\mathbb D}_4$, the simple $A$-module $S_2$ is a proper subrepresentation of any $\theta$-stable module of $\tilde{\mathbb D}_4$ and so $\theta(2)<0$. Viewing $2$ as a source for $\mathbb D_5$, one gets that $\theta(2)>0$ (contradiction). The exact same arguments show that for two $\theta$-stable dimension vectors $\hh_1$ and $\hh_2$ with $\hh_1(2), \hh_2(2) \geq 1$, $\module(\mathbb D_5, \hh_1)^s_{\theta}$ and $\module(\widetilde{\mathbb D}_4, \hh_2)^s_{\theta}$ can not be both non-empty.

The above ``separation'' property for $\theta$-stable irreducible components tells us that $C$ has a (unique) $\theta$-stable decomposition of the form:
$$
C=m_1\cdot C_1 \pp \ldots \pp m_n\cdot C_n,
$$
where $m_1, \ldots, m_n \geq 1$, $C_i \subseteq \module(A,\dd_i)$, $1 \leq i \leq m$, are $\theta$-stable regular irreducible components, and $\dd_i\neq \dd_j$ for all $1 \leq i\neq j \leq n$. 

If the $C_i$'s are orbit closures, with $C_i=\overline{\GL(\dd_i)M_i}$, $\forall 1 \leq i \leq n$, then, for a generic module $M \in C^{ss}_{\theta}$, $\bigoplus_{i=1}^n M_i^{m_1} \simeq \gr_{\theta}(M) \in \overline{\GL(\dd)M}$. This implies that $\bigoplus_{i=1}^n C_i^{\oplus m_i} \subseteq C$. Applying Theorem \ref{theta-stable-decomp-thm}, we conclude that $\M(C)^{ss}_{\theta}$ is a point in this case.

Let us assume now that at least, hence exactly one, of these irreducible components, say $C_1$, is not an orbit closure. That means that $C_1=\module(\widetilde{\mathbb{D}}_4, \delta)$ and, for $2 \leq i \leq n$, $C_i=\module(Q_i, \dd_i)$ with $Q_i$  either $\mathbb{D}_5$ or $\widetilde{\mathbb{D}}_4$, and $\dd_i$ a real Schur root of $Q_i$. (Here, $\delta$ stands for the isotropic Schur root of $\widetilde{\mathbb{D}}_4$.) Furthermore, using the separation property discussed above, if $Q_i$ is $\mathbb{D}_5$ then $\dd_i(2)=0$.

It is not difficult to see that for any Schur $A$-module $M \in C_i$, we have that $\rk M(\beta)=\dd_i(2)$ for all $1 \leq i \leq n$. So, the generic module $M$ in $C$ has a filtration whose factors along the arrow $\beta$ have rank $\dd_i(2)$ with multiplicity $m_i$, where $ 1 \leq i \leq n$. We deduce that $\rk M(\beta) \geq \sum_{i=1}^n m_i\dd_i(2)=\dd(2)$ which implies that $C=C(0,\dd(2))=\module(\widetilde{\mathbb{D}}_4, \dd)$ and, therefore, $\bigoplus_{i=1}^n C_i^{\oplus m_i} \subseteq C$. At this point, we can apply Theorem \ref{theta-stable-decomp-thm} and conclude that
$$
\M(C)^{ss}_{\theta} \simeq S^{m_1}(\M(C_1)^{ss}_{\theta})=\mathbb P^{m_1}.
$$
\end{proof}

Finally, we prove the following result on moduli spaces of modules over stable irreducible components (see also \cite[Proposition 7(2)]{CC13}). 

\begin{prop} \label{prop-Schur-tame-curves} Let $A$ be a Schur-tame algebra, $\dd$ dimension vector of $A$, $\theta$ an integral weight such that $\dd$ is $\theta$-stable, and $C \subseteq \module(A,\dd)$ a $\theta$-stable irreducible component. Then $\M(C)^{ss}_{\theta}$ is either a point or a rational projective curve. 
\end{prop}

\begin{proof} First, we have that $\dim \M(C)^{ss}_{\theta}=\dim C-\dim \GL(\dd)+1$ since $C$ contains $\theta$-stable points.  Hence, $\M(C)^{ss}_{\theta}$ is either a point or a projective curve by Lemma \ref{dim-count-irr-comp}.

Next, let us check that $\M(C)^{ss}_{\theta}$ is rational. If $C$ is an orbit closure, there is nothing to check. So let us assume that 
$$C=\overline{\bigcup_{\lambda \in \mathcal U}\GL(\dd)f(\lambda)}$$
where $(\mathcal U\subseteq k^*,f:\mathcal U \to C)$ is a parameterizing pair. Choose an non-empty open subset $X_0 \subseteq C$ such that $X_0 \subseteq \bigcup_{\lambda \in \mathcal U} \GL(\dd)f(\lambda) \cap C^s_{\theta}$. We can certainly assume that $X_0$ is $\GL(\dd)$-invariant since otherwise we can simply work with $\bigcup_{g \in \GL(\dd)}gX_0$. Set $\mathcal U_0:=\{\lambda \in \mathcal U \mid f(\lambda) \in X_0\}$, which is a non-empty open subset of $\mathcal U$. 

Now, let $\pi: C^{ss}_{\theta}\to \M(C)^{ss}_{\theta}$ be the quotient morphism and let $\varphi:\mathcal U_0 \to \M(C)^{ss}_{\theta}$ be the morphism defined by $\varphi(\lambda)=\pi(f(\lambda)), \forall \lambda \in \mathcal U_0$. If $\lambda_1, \lambda_2 \in \mathcal U_0$ are so that $\varphi(\lambda_1)=\varphi(\lambda_2)$ then $\pi(f(\lambda_1))=\pi(f(\lambda_2))$, which is equivalent to $\overline{\GL(\dd)f(\lambda_1)}\cap \overline{\GL(\dd)f(\lambda_2)}\cap C^{ss}_{\theta}\neq \emptyset$. Since the orbits of $f(\lambda_1)$ and $f(\lambda_2)$ are closed in $C^{ss}_{\theta}$, as they are both $\theta$-stable, we get that $f(\lambda_1)\simeq f(\lambda_2)$ which implies that $\lambda_1=\lambda_2$. 

The injectivity of $\varphi$ together with the fact that $\dim \mathcal U_0=1$ and $\dim \M(C)^{ss}_{\theta} \leq 1$ implies that $\varphi$ is an injective dominant morphism, and hence is birational. This shows that $\M(C)^{ss}_{\theta}$ is a rational projective curve.

\end{proof}

\begin{rmk} It would be interesting to describe those rational projective curves that arise in Proposition \ref{prop-Schur-tame-curves}. In all the examples that we looked at these moduli spaces are nothing else but $\PP^1$'s.

On the other hand, it is known that any wild (equivalently, strictly wild) quasi-tilted or strongly simply connected algebra has a singular moduli space of modules (see \cite{CC10}). So, it is natural to ask whether one always encounters singular moduli spaces of modules for strictly wild algebras.

We plan to address these issues in subsequent work on this subject. 
\end{rmk}

\bibliography{biblio}\label{biblio-sec}
\end{document}